\documentclass[11pt]{amsart}
\usepackage[T1]{fontenc}
\usepackage{amsmath,amssymb,amsthm,mathtools}
\usepackage{booktabs,array}
\usepackage{microtype}
\usepackage[hidelinks]{hyperref}
\theoremstyle{plain}
\newtheorem{theorem}{Theorem}[section]
\newtheorem{proposition}[theorem]{Proposition}
\newtheorem{lemma}[theorem]{Lemma}
\newtheorem{corollary}[theorem]{Corollary}
\theoremstyle{definition}
\newtheorem{definition}[theorem]{Definition}

\numberwithin{equation}{section}

\DeclareMathOperator{\Nm}{N}

\DeclareMathOperator{\ord}{ord}
\newcommand{\Z}{\mathbb Z}
\newcommand{\Q}{\mathbb Q}
\newcommand{\F}{\mathbb F}
\newcommand{\GG}{G}
\newcommand{\pp}{\mathfrak p}
\newcommand{\qq}{\mathfrak q}
\newcommand{\dd}{\mathrm d}
\newcommand{\ee}{\mathrm e}
\title[Integer group determinants for $C_7\times C_7$]
{Integer group determinants for the elementary abelian group of order~49}
\author{Chatchawan Panraksa}
\address{Applied Mathematics Program, Mahidol University International College,
Salaya, Nakhon Pathom 73170, Thailand}
\email{chatchawan.pan@mahidol.ac.th}
\date{}
\hypersetup{pdftitle={Integer group determinants for the elementary abelian group of order 49},pdfauthor={Chatchawan Panraksa}}
\begin{document}
\begin{abstract}
Let $G=C_7\times C_7$, and let $S(G)$ denote the set of integer
values of its group determinant. We determine $S(G)$ by classifying
the values divisible by $7$; the coprime values are already known.
Every nonzero divisible value has valuation at least $10$, and every
multiple of $7^{12}$ occurs. At valuations $10$ and $11$, we give
necessary and sufficient conditions on the cofactor in terms of two
additive invariants of ideals in $\mathbb Z[\zeta_7]$. The first
condition is a bounded signed sum of prime-ideal invariants. The
second requires a prime ideal with nonzero invariant pair whose norm
divides the cofactor. Neither cofactor set is a union of congruence
classes modulo any positive integer. The least positive divisible
value is $43\cdot7^{10}$, and the least positive value of valuation
$11$ is $8\cdot7^{11}$. The proof combines integral reconstruction
from character values with a calculation of the global-unit image
modulo $7$. We conclude by identifying the additional local
conditions and realization problems that arise at primes at least $11$.
\end{abstract}
\keywords{Group determinant, Taussky--Todd problem, cyclotomic norm,
unit group, elementary abelian group}
\subjclass[2020]{Primary 11C20; Secondary 11R18, 15B36}
\maketitle
\section{Introduction}\label{sec:intro}

The integer group determinant problem asks which integers occur as
determinants of the matrices $(v_{gh^{-1}})_{g,h\in G}$ attached to a
finite group $G$. Write
\[
 \Theta_G(v)=\det(v_{gh^{-1}})_{g,h\in G},\qquad
 S(G)=\{\Theta_G(v):v\in\Z^{|G|}\}.
\]
This is the Taussky--Todd problem. For abelian groups, character
factorization expresses $\Theta_G$ as a product of cyclotomic norms.
Complete answers are
known for all groups of order less than $20$
\cite{PinnerSmyth2020,PaudelPinner2025Order18} and for both groups
of order $25$ \cite{MossinghoffPinner2022,Panraksa2026}.
We determine the integer values for $G=C_7\times C_7$.

For an odd prime $p$, the values for $C_p\times C_p$ coprime to
$p$ are characterized by $m^{p-1}\equiv1\pmod{p^2}$.
Trafford~\cite[Theorem~12]{Trafford1992} records the general
elementary abelian criterion, attributing it to Odoni~\cite{Odoni1983};
see also DeSilva and Pinner~\cite[Lemmas~2.1--2.2]{DeSilvaPinner2014}.
The lower bound $p+3$ for the valuation of a nonzero divisible
value, and the occurrence of every valuation at least $p+3$,
appear in Trafford~\cite[Lemma~13]{Trafford1992}; see also
Mossinghoff and Pinner~\cite[Theorem~2.1]{MossinghoffPinnerHeisenberg}.
For $p=3,5$, every cofactor at the lowest valuation occurs, giving
the classifications of Pinner and Smyth
\cite[Theorem~6.1]{PinnerSmyth2020} and of Panraksa
\cite[Corollary~1.2]{Panraksa2026}:
\begin{align*}
 S(C_3^2)&=\{m\equiv\pm1\pmod9\}\cup3^6\Z,\\
 S(C_5^2)&=\{m\equiv\pm1,\pm7\pmod{25}\}\cup5^8\Z.
\end{align*}
At $p=7$, further restrictions appear at the first two valuations,
$10$ and $11$. They depend on the prime factors of the cofactor,
as do the restrictions in Mossinghoff and Pinner's prime-power
circulant classifications~\cite{MossinghoffPinner2022}.
Our criteria distinguish which prime factors can occur in a
cyclotomic norm and how their contributions can be combined.

Two additive invariants of ideals describe these contributions. Put
$K=\Q(\omega)$, $\omega=\zeta_7$, $R=\Z[\omega]$ and
$\pi=1-\omega$. In
$A=R/7R\simeq\F_7[\pi]/(\pi^6)$, put
\[
 \Lambda=\log\omega=-\sum_{j=1}^5\frac{\pi^j}{j}.
\]
For $z\in A^\times$, with constant residue $a=z\bmod\pi$, define
\begin{equation}\label{eq:intro-defects}
 \log(z/a)=\sum_{j=1}^5c_j(z)\Lambda^j,
 \qquad \dd(z)=4c_3(z),\quad\ee(z)=c_5(z).
\end{equation}
The logarithm is the finite polynomial
$\sum_{j=1}^5(-1)^{j+1}X^j/j$ applied to $X=z/a-1$.
Since $K$ has class number one, every ideal coprime to $7$ has a
generator. The invariants defined in \eqref{eq:intro-defects} are unchanged
by multiplying a generator by a unit, and hence define
$\dd(\mathfrak a),\ee(\mathfrak a)$ for such ideals; these facts
are proved below.

For a rational prime $\ell\ne7$, let $f_\ell=\ord_7(\ell)$ and
choose a prime ideal $\qq\mid\ell$. Its norm is $\ell^{f_\ell}$.
Let $\delta(\ell)\in\{0,1,2,3\}$ represent the pair
$\{\pm\dd(\qq)\}$, and call $\ell$ \emph{invisible} if
$(\dd(\qq),\ee(\qq))=(0,0)$, and \emph{visible} otherwise.
Both definitions are independent of
the chosen prime above $\ell$. A prime divisor $\ell$ of
$m=\pm\prod_{\ell\ne7}\ell^{a_\ell}$ is \emph{usable} if
$a_\ell\ge f_\ell$; this means that the norm of a prime ideal above
$\ell$ divides $|m|$. Finally, put
\[
 A_7=\{m\in\Z:m^6\equiv1\pmod{49}\}
     =\{m\equiv\pm1,\pm18,\pm19\pmod{49}\}.
\]

\begin{theorem}\label{thm:main}
For $G=C_7\times C_7$,
\[
 S(G)=A_7\cup7^{10}L_{10}\cup7^{11}L_{11}\cup7^{12}\Z,
\]
where $L_{10}$ and $L_{11}$ contain only nonzero integers coprime
to $7$. For $m=\pm\prod_{\ell\ne7}\ell^{a_\ell}$:
\begin{enumerate}
\item[(i)] $m\in L_{10}$ if and only if integers $j_\ell$ exist with
$|j_\ell|\le\lfloor a_\ell/f_\ell\rfloor$ and
\[
 \sum_\ell j_\ell\delta(\ell)\equiv\pm1\pmod7.
\]
\item[(ii)] $m\in L_{11}$ if and only if a usable prime divisor of
$m$ is not invisible.
\end{enumerate}
Both cofactor sets are symmetric about zero, and
$L_{10}\subsetneq L_{11}$.
\end{theorem}

For example, $\delta(2)=3$, $\delta(29)=2$ and $\delta(43)=1$.
Thus $8,29^2\in L_{11}\setminus L_{10}$, while $43,64\in L_{10}$.
The prime $379$ has $\dd=0$ and $\ee\ne0$, so it belongs to
$L_{11}\setminus L_{10}$ as well. Explicit prime generators are
given in Lemma~\ref{lem:small-prime-types}.

\begin{corollary}\label{cor:minima}
The least positive $7$-divisible determinant value is
$43\cdot7^{10}$. The least positive value of valuation exactly
$11$ is $8\cdot7^{11}$. In particular,
$\pm7^{10},\pm7^{11}\notin S(G)$, whereas $7^{12}\Z\subseteq S(G)$.
\end{corollary}

\begin{corollary}\label{cor:sensitivity}
Neither $L_{10}$ nor $L_{11}$ is a union of congruence classes
modulo any positive integer.
\end{corollary}

The proof has two parts. Integral reconstruction determines which
collections of character values come from integer coefficient vectors.
The image of the global units in $A^\times$ then determines which
residues can occur for a fixed cofactor ideal. Together these give
necessary and sufficient conditions at valuations $10$ and $11$,
which we translate into the prime-factor criteria above.
The Chinese remainder theorem gives integers in the same residue
class with different cofactor membership, proving
Corollary~\ref{cor:sensitivity}.
The final section explains which parts of the argument require new
ideas when $p\ge11$.


\section{Character values and local conditions}\label{sec:prelim}

We first identify the possible character valuations at the two
exceptional determinant valuations, then give the local conditions
needed to construct them.
Write $\pp=\pi R$. Then $7R=\pp^6$ and $R/\pp=\F_7$.
The conjugate embeddings $\sigma_t(\omega)=\omega^t$ pair off, so
nonzero elements of $R$ have positive norm and units have norm one.
Also $h_K=1$: the discriminant has absolute value $7^5$, giving
Minkowski bound $<5$, while the prime ideals above $2$ and $3$ have
norms $8$ and $729$. Thus $R$ is the only integral ideal below the
bound. The discriminant and splitting formulas are given in
Neukirch~\cite[pp.~59--61]{Neukirch1999}; for further cyclotomic
background, see Washington~\cite[Chapter~2, pp.~9--19]{Washington1997}.
We use $v_7(7)=v_\pp(\pi)=1$ and $v(0)=\infty$.

Label $\GG=\F_7^2$ by $(i,j)$ and its characters by
$\chi_{a,b}(i,j)=\omega^{ai+bj}$. Put $s(v)=\sum v_{i,j}$ and
$f_{a,b}(v)=\sum v_{i,j}\omega^{ai+bj}$. The $48$ nontrivial
characters form eight Galois orbits indexed by $\mathbb P^1(\F_7)$,
with representatives $(0,1)$ and $(1,b)$, $b\in\F_7$. We call these
directions $\infty$ and $b$, respectively. Character factorization gives
\begin{equation}\label{eq:factorization}
 \Theta_\GG(v)=s(v)\prod_rN_r(v),\qquad
 N_r(v)=\Nm_{K/\Q}(f_r(v))\in\Z_{\ge0}.
\end{equation}

\begin{lemma}\label{lem:chi-mod-p}\label{lem:val-norm}
For every character $\chi$ and nonzero $\alpha\in R$,
\[
 \chi(v)\equiv s(v)\pmod\pp,\qquad
 v_7(\Nm_{K/\Q}\alpha)=v_\pp(\alpha).
\]
\end{lemma}
\begin{proof}
The congruence follows from $\omega\equiv1\pmod\pp$.
For the valuation identity, use
$v_\pp(\Nm\alpha)=6v_\pp(\alpha)$ and $v_\pp(7)=6$.
\end{proof}
Consequently, for nonzero determinants,
\begin{equation}\label{eq:v7-formula}
 v_7(\Theta_\GG(v))=v_7(s(v))+\sum_rv_\pp(f_r(v)).
\end{equation}

For $\ell_\infty(i,j)=j$ and $\ell_b(i,j)=i+bj$, define the
\emph{ray vector} $c^{(r)}\in\Z^7$ by
$c^{(r)}(m)=\sum_{\ell_r(i,j)=m}v_{i,j}$. Each has sum $s(v)$, and
\begin{equation}\label{eq:ray}
 f_r(v)=\sum_{m\in\F_7}c^{(r)}(m)\omega^m.
\end{equation}
More generally, $f\in R$ and a prescribed sum $s\equiv f\pmod\pp$
determine a unique ray vector: coefficient vectors for $f$ differ
by multiples of $(1,\ldots,1)$, and the sum fixes that multiple.
Its reduction $\bar c\in\F_7^7$ is the \emph{ray residue} of $(f,s)$.
Write $\bar c(m)=\sum_{d=0}^6\gamma_dm^d$ for its unique polynomial
representation, with $0^0=1$.

Fixing the sum also fixes the constant coordinate $\gamma_0$;
the next lemma records the precision this requires.
\begin{lemma}\label{lem:cbar-window}
Write $f=\sum_{i=0}^5a_i\omega^i$ and $\ell(f)=\sum_i a_i$.
The ray vector of $(f,s)$ is
\[
 c=(a_0,\ldots,a_5,0)+t(1,\ldots,1),\qquad
 t=(s-\ell(f))/7\in\Z.
\]
Thus $\bar c$ depends only on $(f\bmod7R,t\bmod7)$, hence on
$(f\bmod49R,s\bmod49)$; its nonconstant coordinates depend only on
$f\bmod7R$. With $s$ fixed, replacing $f$ by $f'=f+7g$ gives
$\bar c'=\bar c-\overline{\ell(g)}(1,\ldots,1)$. In particular,
$f'-f\in7\pp$ implies $\ell(f')\equiv\ell(f)\pmod{49}$.
\end{lemma}
\begin{proof}
The displayed vector evaluates to $f$ and has sum $s$.
The coefficient $t$ is integral since
$\ell(f)\equiv f\equiv s\pmod\pp$ and $\pp\cap\Z=7\Z$.
Reduction gives the precision assertions. Replacing $f$ by $f+7g$
changes $t$ by $-\ell(g)$ and the other coefficients by multiples
of seven, hence changes only $\gamma_0$. If $g\in\pp$, then
$\ell(g)\equiv0\pmod7$, proving the last assertion.
\end{proof}

Expanding $\omega^m=(1-\pi)^m$ gives
\begin{equation}\label{eq:digits}
 f=\sum_{j=0}^6(-1)^j\mathcal M_j(c)\pi^j,\qquad
 \mathcal M_j(c)=\sum_{m=0}^6\binom mj c(m)\in\Z.
\end{equation}
The power sums on $\F_7$ give the moment dictionary
\begin{equation}\label{eq:moment-dict}
 \mathcal M_0\equiv-\gamma_6,\qquad
 \mathcal M_1\equiv-\gamma_5,\qquad
 \mathcal M_2\equiv\tfrac12(\gamma_5-\gamma_4)\pmod7.
\end{equation}

\begin{lemma}\label{lem:digits}
For $0\le k\le6$, $v_\pp(f)\ge k$ if and only if
$7\mid\mathcal M_j(c)$ for every $j<k$.
\end{lemma}
\begin{proof}
Under these divisibilities, every term of the expansion
\eqref{eq:digits} lies in $\pp^k$.
Otherwise the least $j<k$ with $7\nmid\mathcal M_j$ supplies
the unique term of valuation $j$: earlier terms have valuation at
least six, and later terms have larger valuation. Hence $v_\pp(f)=j$.
\end{proof}
Thus valuations at least $1,2,3$ correspond, respectively, to the
vanishing of $\gamma_6$, of $\gamma_6,\gamma_5$, and of
$\gamma_6,\gamma_5,\gamma_4$. For a valuation-one ray put
$\lambda=\mathcal M_1\bmod7=-\gamma_5$, its \emph{leading residue};
then $f\equiv-\lambda\pi\pmod{\pp^2}$.

There are $49$ character factors, so $\Theta_\GG(-v)=-\Theta_\GG(v)$.
Nontrivial characters annihilate $\mathbf1_\GG$, giving
\begin{equation}\label{eq:shift}
 \Theta_\GG(v+c\mathbf1_\GG)
   =(s(v)+49c)\prod_rN_r(v)\qquad(c\in\Z).
\end{equation}

\begin{lemma}[Normalization]\label{lem:normalize}
\begin{enumerate}
\item[(i)] A change of group coordinates by $\varphi\in\mathrm{GL}_2(\F_7)$
preserves the determinant and augmentation, and permutes the eight
orbit valuations. Each chosen direction can be moved to $\infty$.
\item[(ii)] The scalar change $\varphi_\alpha(i,j)=(\alpha i,\alpha j)$
fixes each orbit and sends $f_{a,b}$ to $\sigma_\alpha(f_{a,b})$,
multiplying its leading residue by $\alpha$. Thus a common nonzero
leading residue may be normalized to one.
\end{enumerate}
\end{lemma}
\begin{proof}
A change of coordinates permutes the character factors and acts
transitively on their projective directions; Galois twists preserve
valuations. For a scalar change,
$f_{a,b}(v\circ\varphi_\alpha^{-1})=f_{\alpha a,\alpha b}(v)
=\sigma_\alpha(f_{a,b}(v))$. Its ray permutation is
$\bar c(m)\mapsto\bar c(\alpha^{-1}m)$, so
$\lambda$ changes to $\alpha^{-5}\lambda=\alpha\lambda$.
\end{proof}

\begin{theorem}[Known coprime criterion]\label{thm:coprime}
The determinant values coprime to $7$ are precisely
\[
 S_1(\GG)=A_7=\{m\in\Z:m^6\equiv1\pmod{49}\}.
\]
\end{theorem}
\begin{proof}
Apply DeSilva and Pinner's result
\cite[Lemmas~2.1--2.2]{DeSilvaPinner2014} with $(p,n)=(7,2)$.
\end{proof}

\subsection{The two exceptional valuations}\label{sec:valuation}

\begin{lemma}[First moments]\label{lem:second-order}\label{lem:nonreal}
Suppose that $7\mid s(v)$. For $(a,b)\ne(0,0)$ put
$L_v(a,b)=\sum v_{i,j}(ai+bj)$. Then
\[
 f_{a,b}(v)\equiv-\pi L_v(a,b)\pmod{\pp^2}.
\]
Thus the leading residue is $L_v(a,b)\bmod7$.
If this linear form is nonzero modulo seven, then exactly one projective
direction has character valuation at least two; otherwise all eight do.
\end{lemma}
\begin{proof}
Expand $\omega^k\equiv1-k\pi\pmod{\pp^2}$ and use
$s(v)\in\pp^6$. A nonzero linear form on $\F_7^2$ has one kernel line.
\end{proof}

\begin{theorem}\label{thm:noncoprime}
If $7\mid\Theta_\GG(v)\ne0$, then $v_7(\Theta_\GG(v))\ge10$.
\end{theorem}
\begin{proof}
Lemma~\ref{lem:chi-mod-p} forces $7\mid s(v)$.
The first-moment calculation in Lemma~\ref{lem:second-order}
gives $\sum_rv_\pp(f_r)\ge9$ if $L_v\ne0\pmod7$, and at least
$16$ otherwise. The valuation formula \eqref{eq:v7-formula}
therefore gives the bound.
\end{proof}
This is the case $p=7$ of the bound in
Trafford~\cite[Lemma~13]{Trafford1992} and
Mossinghoff and Pinner~\cite[Theorem~2.1]{MossinghoffPinnerHeisenberg}.
The same argument determines all profiles at valuations ten and eleven.

\begin{lemma}[Valuation profiles]\label{lem:pinning}
Suppose that $D=\Theta_\GG(v)\ne0$. If $v_7(D)=10$, then $v_7(s)=1$
and the orbit valuations have profile $(1^7,2)$. If $v_7(D)=11$,
then either $v_7(s)=2$ with profile $(1^7,2)$, or $v_7(s)=1$
with profile $(1^7,3)$. In each case
\[
 D/7^{v_7(D)}=(s/7^{v_7(s)})\prod_rm_r,\qquad
 m_r=N_r/7^{v_\pp(f_r)},
\]
where each $m_r$ is a positive norm from $R$ coprime to seven.
\end{lemma}
\begin{proof}
If all orbit valuations were at least two, their sum would be at
least $16$. Otherwise seven are one, and the eighth is at least two.
The valuation formula \eqref{eq:v7-formula} gives the stated possibilities.
Also $m_r=\Nm(f_r/\pi^{v_\pp(f_r)})$, since $\Nm\pi=7$.
\end{proof}
Call the valuation-one directions \emph{generic} and the remaining
one \emph{special}. Move the latter to $\infty$. Then $L_v(0,1)=0$
and the seven values $L_v(1,b)=L_v(1,0)\ne0$ coincide.
Lemma~\ref{lem:normalize}(ii) normalizes this common residue to one.


\subsection{Integral reconstruction}\label{sec:method}

To construct a determinant, we choose its eight character values
and ask whether they come from an integer vector.
Their ray residues give the exact integrality condition.

\begin{lemma}[Reconstruction]\label{lem:interp}
Let $c^{(r)}\in\Z^7$ have common coordinate sum $s$, and put
\[
 F(i,j)=c^{(\infty)}(j)+\sum_{b\in\F_7}c^{(b)}(i+bj).
\]
\begin{enumerate}
\item[(i)] If the $c^{(r)}$ are the ray vectors of $v$, then
$F(i,j)=7v_{i,j}+s$.
\item[(ii)] Given $f_r\in R$ with $f_r\equiv s\pmod\pp$, take their
ray vectors of sum $s$. An integral vector with character data
$(s;\sigma_t(f_r))$ exists if and only if $F\equiv s\pmod7$ on
$\F_7^2$. When it exists, it is $v=(F-s)/7$.
\end{enumerate}
\end{lemma}
\begin{proof}
For (i), the cell $(i,j)$ is counted eight times in $F(i,j)$, and
every other cell once.
For (ii), necessity follows from (i).
To prove sufficiency, sum $F$ on a line $\ell_r(i,j)=m$.
The $r$-summand contributes
$7c^{(r)}(m)$ to $\sum F$; each of the other seven directions takes
every value once and contributes $s$. Hence
\[
 \sum_{\ell_r(i,j)=m}(F(i,j)-s)/7=c^{(r)}(m).
\]
Thus $(F-s)/7$ has every prescribed ray vector and augmentation.
\end{proof}

We solve this integrality condition one polynomial degree at a time.
\begin{lemma}[Degree decomposition]\label{lem:degree}
Write $h_d(b)=\gamma_d(\bar c^{(b)})$ and
$\gamma_{\infty,d}=\gamma_d(\bar c^{(\infty)})$.
The function $\bar F=\sum_r\bar c^{(r)}\circ\ell_r$ is constant
if and only if, for every $d=1,\ldots,6$,
\begin{enumerate}
\item[(K$_d$)] $\sum_bh_d(b)b^k=0$ for $0\le k<d$, and
$\gamma_{\infty,d}=-\sum_bh_d(b)b^d$.
\end{enumerate}
Equivalently, $h_d$ is a polynomial function of degree at most
$6-d$, and $\gamma_{\infty,d}$ is the displayed negative moment.
The constant value is then $\sum_r\gamma_0(\bar c^{(r)})$.
\end{lemma}
\begin{proof}
Evaluation is injective on polynomials of degree at most six in
each variable. Thus each homogeneous component of positive degree
in $\sum_r\bar c^{(r)}(\ell_r(x,y))$ must vanish.
For degree $d$, the coefficients of $x^{d-k}y^k$ are
$\binom dk\sum_bh_d(b)b^k$ for $k<d$, and
$\sum_bh_d(b)b^d+\gamma_{\infty,d}$ for $k=d$.
The binomial coefficients are nonzero modulo seven.
Finally, the first $d$ moments vanish exactly when the coefficients
of $b^6,b^5,\ldots,b^{7-d}$ in $h_d$ vanish, by the power sums on
$\F_7$.
\end{proof}

Once the nonconstant terms vanish, only a condition on the sum remains.
For $s=7\sigma_0$, first normalize all local ray vectors to sum
seven. Changing their sum to $s$ adds $\sigma_0-1$ to each
$\gamma_0$, by Lemma~\ref{lem:cbar-window}. After the nonconstant
conditions hold, $\bar F\equiv s=0\pmod7$ therefore becomes
\begin{equation}\label{eq:kappa}
 \sum_r\gamma_0(\bar c^{(r)})=\kappa,
 \qquad \kappa=1-\sigma_0\pmod7.
\end{equation}
Thus $\kappa=1$ exactly when $49\mid s$.
For every $s^*\equiv s\pmod{49}$, replacing $s$ by $s^*$ leaves
all ray residues unchanged. Hence every feasible configuration
realizes that full progression of sums.

The next two identities reduce the degree-two conditions after
the ray-coordinate formulas are substituted.
\begin{lemma}\label{lem:identity}
Let $x_3,x_4\in\F_7[b]$ have degrees at most $3,2$, respectively,
and write $q_i=[b^i]x_4$ and $x_{33}=[b^3]x_3$. Then
\begin{align*}
 \sum_b(-x_4(b)^3-2x_3(b)x_4(b))&=q_2^3,\\
 \sum_b b(-x_4(b)^3-2x_3(b)x_4(b))
     &=3q_1q_2^2+2x_{33}q_2.
\end{align*}
\end{lemma}
\begin{proof}
For a polynomial of degree at most seven, its sum on $\F_7$ is
minus its $b^6$ coefficient. In the first identity only $x_4^3$
contributes. In the second, the relevant coefficients are
$[b^5]x_4^3=3q_1q_2^2$ and $[b^5]x_3x_4=x_{33}q_2$.
\end{proof}


\subsection{The unit image and ray residues}\label{sec:defect}

Generators of a fixed cofactor ideal differ by a global unit.
We determine their possible ray residues by first computing the
image of these units in the local ring.

\subsubsection{Finite logarithms and global units}

Put $A=R/7R$ and $I=\pi A$.  Since $(7)=\pp^6$, we have
$A\cong\F_7[\pi]/(\pi^6)$.  Define
\[
 \log(1+X)=\sum_{j=1}^{5}\frac{(-1)^{j+1}X^j}{j},\qquad
 \exp(X)=\sum_{j=0}^{5}\frac{X^j}{j!}\qquad(X\in I).
\]
All denominators are invertible. The formal identities through degree
five give mutually inverse group isomorphisms
$\log\colon1+I\longrightarrow I$ and $\exp\colon I\longrightarrow1+I$.
These are finite operations in characteristic seven. In particular,
\[
 \Lambda=\log\omega=-\pi-\pi^2/2-\pi^3/3-\pi^4/4-\pi^5/5.
\]
Its linear coefficient is invertible, so $A=\F_7[\Lambda]/(\Lambda^6)$
and $\pi=1-\exp\Lambda$.

For $z\in A^\times$, let $a=z\bmod I\in\F_7^\times$, viewed as a
constant in $A$, and write
\begin{equation}\label{eq:finite-log}
 \mathcal L(z):=\log(a^{-1}z)
       =\sum_{j=1}^{5}c_j(z)\Lambda^j.
\end{equation}
The map $\mathcal L\colon A^\times\to I$ is a surjective homomorphism with
kernel $\F_7^\times$.  For $t\in\F_7^\times$, the automorphism
$\sigma_t(\omega)=\omega^t$ fixes the constant field and satisfies
$\sigma_t\Lambda=t\Lambda$.  Consequently
\begin{equation}\label{eq:log-characters}
 c_j(zz')=c_j(z)+c_j(z'),\qquad
 c_j(\sigma_tz)=t^jc_j(z).
\end{equation}

The third and fifth coefficients measure the obstruction to
being the residue of a global unit.
\begin{definition}\label{def:de}
For $z\in R$ coprime to $\pi$, define its two \emph{defects} by
\begin{equation}\label{eq:log-defects}
 \dd(z)=4c_3(z\bmod7R),\qquad \ee(z)=c_5(z\bmod7R).
\end{equation}
All defect values and subsequent ray-coordinate formulas lie in $\F_7$.
\end{definition}

\begin{lemma}[The global-unit image]\label{lem:unit-images}
Let $U=\langle\Lambda,\Lambda^2,\Lambda^4\rangle_{\F_7}$.  The image of
$R^\times$ in $A^\times$ is
\[
 \overline{R^\times}=\F_7^\times\exp(U).
\]
If $\overline{R^\times}(k)$ denotes its image modulo $\pp^k$, then
\[
 |\overline{R^\times}(4)|=6\cdot7^2=294,\qquad
 |\overline{R^\times}(5)|=|\overline{R^\times}(6)|
       =6\cdot7^3=2058.
\]
The logarithm induces the isomorphisms
\[
 A^\times/\overline{R^\times}
  \cong I/U\cong\F_7\Lambda^3\oplus\F_7\Lambda^5.
\]
\end{lemma}

\begin{proof}
We first show that global units have zero third and fifth
logarithmic coefficients.
For a global unit $u$, the quotient $u/\bar u$ is an algebraic
unit whose conjugates all have absolute value one.
By Kronecker's theorem
\cite[Chapter~I, Proposition~7.1]{Neukirch1999}, this quotient is a
root of unity, hence $\pm\omega^a$. Its residue modulo $\pp$
is one, so $u/\bar u=\omega^a$. Since conjugation sends
$\Lambda$ to $-\Lambda$, the identity
$\mathcal L(u)-\mathcal L(\bar u)=a\Lambda$ forces $c_3(u)=c_5(u)=0$.
Therefore $\overline{R^\times}\subseteq\F_7^\times\exp(U)$.

For the reverse inclusion, use the explicit units
$\omega,1+\omega$ and $1+\omega+\omega^2$. The last two are units
because they are the quotients $(1-\omega^2)/(1-\omega)$ and
$(1-\omega^3)/(1-\omega)$, respectively, and numerator and denominator
generate the same prime ideal. Finite expansion gives
\begin{align*}
 \mathcal L(\omega)&=\Lambda,\\
 \mathcal L(1+\omega)
   &=\frac{\Lambda}{2}+\frac{\Lambda^2}{8}
                   -\frac{\Lambda^4}{192}
     =4\Lambda+\Lambda^2+2\Lambda^4,\\
 \mathcal L(1+\omega+\omega^2)
   &=\Lambda+\frac{\Lambda^2}{3}-\frac{\Lambda^4}{36}
     =\Lambda+5\Lambda^2+6\Lambda^4.
\end{align*}
The even-coordinate determinant is $1\cdot6-2\cdot5=3\ne0$, so
these logarithms span $U$. Integer powers realize every linear
combination. Also $(1+\omega+\omega^2)^7=3$ in $A$, and $3$
generates $\F_7^\times$. Removing the constant residue therefore
realizes every $\exp(X)$, $X\in U$. Reduction modulo
$\pp^k=(\Lambda^k)$ gives the image sizes and quotient.
\end{proof}

\begin{theorem}[Defect structure]\label{thm:de-props}
\begin{enumerate}
\item[(a)] The exact conductor exponents of $\dd$ and $\ee$ are $4$ and
$6$, respectively: $\dd$ depends only on $z\bmod\pp^4$ and $\ee$ only
on $z\bmod\pp^6$, and neither factors through a smaller power.
\item[(b)] The defects are additive on products and vanish on $R^\times$.
\item[(c)] For every $t\in\F_7^\times$,
\[
 \dd(\sigma_tz)=\Bigl(\frac{t}{7}\Bigr)\dd(z),\qquad
 \ee(\sigma_tz)=t^{-1}\ee(z).
\]
\item[(d)] The map $(\dd,\ee)\colon A^\times\to\F_7^2$ is surjective and
has kernel exactly $\overline{R^\times}$.
\end{enumerate}
Since $h_K=1$, the defects are consequently well defined on ideals
coprime to $7$, additive on ideal products, equivariant as in~(c), and
zero on ideals generated by rational integers coprime to~$7$.
\end{theorem}

\begin{proof}
The coefficient $c_j$ depends only on the residue modulo
$\pp^{j+1}=(\Lambda^{j+1})$. The classes $\exp(\Lambda^3)$ and
$\exp(\Lambda^5)$ have defect pairs $(4,0)$ and $(0,1)$ and lie
in $1+\pp^3$ and $1+\pp^5$, respectively. This proves exactness
and surjectivity. Equation~\eqref{eq:log-characters} gives additivity
and equivariance, since $t^3=(\frac t7)$ and $t^5=t^{-1}$.
Lemma~\ref{lem:unit-images} identifies the common kernel. Ideal
generators are unique up to units, and rational integers coprime
to seven reduce to constants, whose logarithm is zero.
\end{proof}

Since $h_K=1$, the quotient $A^\times/\overline{R^\times}$ is also
the ray class group of modulus $7R$. Its abstract structure is given
by Hoelscher~\cite[Theorem~1.4]{Hoelscher2010}; the explicit unit
image above is needed to realize the character values.

\subsubsection{Ray coordinates}

For $r=1,2,3$, let $\bar c_r(z)$ be the ray residue of $(\pi^rz,7)$,
with polynomial coordinates $\gamma_d$.
The prescribed sum fixes $\gamma_0$, so we retain the full ray residue
when converting the unit image to ray coordinates.

\begin{lemma}[Exact ray windows]\label{lem:ray-exact-window}
For $r=1,2,3$,
\[
 \bar c_r(z)=\bar c_r(z')
       \quad\Longleftrightarrow\quad z\equiv z'\pmod{\pp^{7-r}}.
\]
The resulting maps are affine bijections from $R/\pp^{7-r}$ onto,
respectively, the coordinate spaces
\[
 \{\gamma_6=0\},\qquad
 \{\gamma_6=\gamma_5=0\},\qquad
 \{\gamma_6=\gamma_5=\gamma_4=0\}.
\]
Local units correspond to the additional condition
$\gamma_{6-r}\ne0$.
\end{lemma}

\begin{proof}
Cyclotomic evaluation identifies integer ray vectors of sum zero
with $\pp$. Two normalized vectors have the same residue precisely
when their difference is seven times such a vector, equivalently
when their values differ by $7\pp=\pp^7$. Dividing by $\pi^r$
gives the exact window. Subtracting the vector of $(0,7)$ makes
the map $\F_7$-linear. Lemma~\ref{lem:digits} places the image in
the stated space, which has the same dimension $7-r$ as the domain.
Injectivity therefore gives bijectivity, and the digit dictionary
identifies local units by the first nonzero coordinate.
\end{proof}

To compute these coordinates, write uniquely in the relevant quotient
\[
 z=a\bigl(1+b_1\pi+\cdots+b_{6-r}\pi^{6-r}\bigr),\qquad
 a\in\F_7^\times,\quad b_i\in\F_7,
\]
and put $S_j=1+\sum_{i=1}^j b_i$.  The ray coordinates are
\begin{equation}\label{eq:ray-substitution}
\begin{array}{cccc}
\toprule
 &r=1&r=2&r=3\\
\midrule
 \gamma_0&1+aS_5&1+aS_4&1+aS_3\\
 \gamma_1&a(1+5b_1+3b_2+5b_3+b_4)
         &a(5+3b_1+5b_2+b_3)&a(3+5b_1+b_2)\\
 \gamma_2&a(1+b_2+4b_3)&a(b_1+4b_2)&a(1+4b_1)\\
 \gamma_3&a(1+b_1+6b_2)&a(1+6b_1)&6a\\
 \gamma_4&a(1+5b_1)&5a&0\\
 \gamma_5&a&0&0\\
 \gamma_6&0&0&0\\
\bottomrule
\end{array}.
\end{equation}
To obtain the table, expand $\pi^rz$ in the $\omega$-basis and apply
Lemma~\ref{lem:cbar-window}.
For the constant row, the coefficient
sums of $\pi,\ldots,\pi^5$ are zero, whereas
\[
 \pi^6=-7\omega+14\omega^2-21\omega^3+14\omega^4-7\omega^5
\]
has coefficient sum $-7$. Normalization to sum seven gives the
stated $\gamma_0$.

\begin{proposition}[The defects in ray coordinates]\label{prop:ray-defects}
Let $z\in R$ be coprime to $\pi$, with leading residue
$\lambda=-\gamma_5=1$ in $\bar c_1(z)$. Then
\begin{align}\label{eq:ray-defects}
 \dd(z)&=\gamma_2+\gamma_4^3+2\gamma_3\gamma_4,\\
 \ee(z)&=\gamma_0-1-3\gamma_4^5-5\gamma_3\gamma_4^3
          -4\gamma_1\gamma_4-3\dd(z)\gamma_4^2-2\dd(z)\gamma_3.
          \notag
\end{align}
\end{proposition}

\begin{proof}
Write
$L_j=[T^j]\sum_{k=1}^5(-1)^{k+1}(b_1T+\cdots+b_5T^5)^k/k$,
so that $\log(z/a)=\sum_{j=1}^5L_j\pi^j$.
Substitution of $\pi=1-\exp\Lambda$ gives
\begin{equation}\label{eq:log-coordinate-conversion}
 c_3=L_1+L_2-L_3,\qquad
 c_5=-L_1+2L_2+4L_3+2L_4-L_5.
\end{equation}
The $\Lambda^3$ and $\Lambda^5$ coefficients of
$\pi,\ldots,\pi^5$ are $(1,1,-1,0,0)$ and $(-1,2,4,2,-1)$.
Let $P_j=\gamma_j/a$ for $j=1,2,3,4$, using the $r=1$ column of
the coordinate table \eqref{eq:ray-substitution}, and let
$D=4(L_1+L_2-L_3)$.
Substituting the expressions for $P_j$ gives the polynomial identities
\[
 -P_2-P_4^3+2P_3P_4=D,
\]
\begin{multline*}
 -S_5+3P_4^5-5P_3P_4^3-4P_1P_4-3DP_4^2+2DP_3\\
 =-L_1+2L_2+4L_3+2L_4-L_5.
\end{multline*}
Expanding $D$ gives
$D=b_1^3-2b_1^2-3b_1b_2-3b_1-3b_2+3b_3$; both identities
follow by collecting coefficients in these finite polynomials.
The normalization $\lambda=1$ gives $a=-1$ in the coordinate table
\eqref{eq:ray-substitution}.
Substituting this value proves the defect formulas \eqref{eq:ray-defects}.
\end{proof}

\begin{theorem}[Ray residues of unit multiples]\label{thm:twisted}
Fix $y\in R$ coprime to $\pi$, and put $d=\dd(y)$, $e=\ee(y)$.
As $u$ ranges over $R^\times$, the ray residues $\bar c_r(yu)$,
normalized to sum seven, have the following exact descriptions.
\begin{enumerate}
\item[(a)] On the slice $\lambda=1$, the residues $\bar c_1(yu)$
are precisely the solutions with $\gamma_6=0$, $\gamma_5=-1$ of
\begin{align*}
 \gamma_2&=-\gamma_4^3-2\gamma_3\gamma_4+d,\\
 \gamma_0&=1+3\gamma_4^5+5\gamma_3\gamma_4^3+4\gamma_1\gamma_4
                  +3d\gamma_4^2+2d\gamma_3+e.
\end{align*}
The coordinates $(\gamma_1,\gamma_3,\gamma_4)$ range freely over $\F_7^3$.
\item[(b)] The residues $\bar c_2(yu)$ are precisely the solutions with
$\gamma_6=\gamma_5=0$ and $\gamma_4\ne0$ of
\[
 \gamma_1=-\gamma_3^3\gamma_4^4
                -3\gamma_2\gamma_3\gamma_4^5+d\gamma_4.
\]
The coordinates $(\gamma_0,\gamma_2,\gamma_3,\gamma_4)$ range freely
over $\F_7^3\times\F_7^\times$.
\item[(c)] The residues $\bar c_3(yu)$ are precisely the solutions with
$\gamma_6=\gamma_5=\gamma_4=0$ and $\gamma_3\ne0$ of
\begin{equation}\label{eq:twisted-third}
 4\gamma_3^2(\gamma_0-1)+\gamma_1\gamma_2\gamma_3
              -\gamma_2^3-d\gamma_3^3=0.
\end{equation}
Equivalently,
\begin{equation}\label{eq:twisted-third-solved}
 \gamma_0=1+2d\gamma_3+2\gamma_2^3\gamma_3^{-2}
                              -2\gamma_1\gamma_2\gamma_3^{-1}.
\end{equation}
The coordinates $(\gamma_1,\gamma_2,\gamma_3)$ range freely over
$\F_7^2\times\F_7^\times$.  At every fixed
$(\gamma_2,\gamma_3)$ there are exactly seven residues.
\end{enumerate}
Thus the residue possibilities of $\pi^ryu$, for $r=1,2,3$, depend on
the ideal $(y)$ only through its defect pair; for $r=2,3$ only $d$ is
needed.
\end{theorem}

\begin{proof}
We first derive the equations satisfied by the unit multiples.
The equations in (a) follow from Proposition~\ref{prop:ray-defects},
since $\dd(yu)=d$ and $\ee(yu)=e$.
For (b) and (c), use the corresponding columns of the coordinate
table \eqref{eq:ray-substitution} and the expression
$D=4(L_1+L_2-L_3)$ in the proof of that proposition.
Substitution, with $a^6=1$,
gives the identities
\begin{align*}
 \gamma_1+\gamma_3^3\gamma_4^4
                 +3\gamma_2\gamma_3\gamma_4^5&=D\gamma_4
                    &&(r=2),\\
 4\gamma_3^2(\gamma_0-1)+\gamma_1\gamma_2\gamma_3-\gamma_2^3
                 &=D\gamma_3^3&&(r=3).
\end{align*}
Since $D=\dd(yu)=d$, these are (b) and (c). Solving (c) uses
$4^{-1}=2$ in $\F_7$.

It remains to show that every permitted residue is attained.
By Lemma~\ref{lem:unit-images}, the unit orbits modulo $\pp^6$
are exactly the fibres of $(\dd,\ee)$. Modulo $\pp^5$ and $\pp^4$ they are the
fibres of $\dd$, because the quotient by the unit image has only
the coordinate $\Lambda^3$. Lemma~\ref{lem:ray-exact-window}
transports these fibres bijectively to ray residues at the exact
windows $\pp^6,\pp^5,\pp^4$, respectively. The displayed identities
specify precisely these defects, with $\lambda=1$ in (a) and
nonzero leading coordinate in (b), (c); thus every permitted tuple
is attained. Solving for
$\gamma_2,\gamma_0$ in (a), $\gamma_1$ in (b), and $\gamma_0$ in
(c) gives all stated free coordinates.
\end{proof}


\subsection{Four prime elements}

Four prime elements suffice for the examples in the introduction
and the proof of the minima.

\begin{lemma}\label{lem:small-prime-types}
The elements in the following table generate prime ideals of $R$.
\[
\begin{array}{@{}rlrrr@{}}
\toprule
 \ell&y_\ell&\Nm(y_\ell)&\dd(y_\ell)&\ee(y_\ell)\\\midrule
 2&1+\omega+\omega^3&8&3&0\\
 29&1+\omega-\omega^2&29&2&1\\
 43&1+2\omega&43&6&6\\
 379&1+2\omega+3\omega^2&379&0&4\\
\bottomrule
\end{array}
\]
Consequently $\delta(2)=3$, $\delta(29)=2$, $\delta(43)=1$,
and $379$ is visible with zero first defect.
\end{lemma}
\begin{proof}
The norms are the products of the six conjugates, using
$1+\omega+\cdots+\omega^6=0$; for instance
$\Nm(1+2\omega)=(1+2^7)/3=43$.
The last three norms are rational primes. Every prime ideal above $2$
has norm $2^3$, so the first norm also forces a prime ideal.
For the defect entries, substitute $\omega=\exp\Lambda$ in
$\F_7[\Lambda]/(\Lambda^6)$ and expand
$\log(y_\ell/y_\ell(1))$. The coefficient pairs
$(c_3,c_5)$ are respectively $(6,0),(4,1),(5,6),(0,4)$.
The defect definitions~\eqref{eq:intro-defects} give the last two columns.
\end{proof}

The least positive divisible value is attained by the indicator of
\[
 T=\{(0,0),(1,5),(2,0),(2,6),(3,4),(4,6),(6,5)\}.
\]
Its coordinate sum is seven, and the eight character norms, in the
order $\infty,0,1,\ldots,6$, are
$(7,7,7,49,7,301,7,7)$. These follow by multiplying the six
conjugates of each ray sum. The factorization~\eqref{eq:factorization} gives
$\Theta_G(\mathbf1_T)=43\cdot7^{10}$.


\section{Criteria for the exceptional profiles}\label{sec:criteria}

We now solve the interpolation conditions for prescribed cofactor
ideals. Normalize the special orbit to $\infty$ and the generic
leading residue to $\lambda=1$. For a profile $(1^7,e_\infty)$, let
$y_b=f_b/\pi$ and $y_\infty=f_\infty/\pi^{e_\infty}$, and prescribe
their cofactor ideals $\mathfrak a_r=(y_r)$, coprime to~$7$. Put
\[
 d_b=\dd(\mathfrak a_b),\quad e_b=\ee(\mathfrak a_b),\quad
 d_\infty=\dd(\mathfrak a_\infty).
\]
The nonconstant conditions involve the first-defect moments
$D_k$, while the constant condition also involves the second-defect
sum $E_0$:
\[
 D_k=\sum_{b\in\F_7}b^kd_b\ (0\le k\le3),\quad E_0=\sum_be_b.
\]
We solve the nonconstant interpolation conditions first and then
the constant condition $\sum_r\gamma_{r,0}=\kappa$, where
$\kappa=1-s/7\pmod7$ is the sum class in \eqref{eq:kappa}.
Fullness in Theorem~\ref{thm:twisted} ensures that unit multiples
realize every permitted coordinate choice.

\begin{theorem}[The profile $(1^7,2)$]\label{thm:crit10}
For prescribed cofactor ideals and each $\kappa\in\F_7$, the profile
$(1^7,2)$ with generic leading residue $1$ and coordinate-sum class
$\kappa$ is realizable if and only if
\[
 D_0=\sum_b\dd(\mathfrak a_b)=\pm1\quad\text{in }\F_7.
\]
When this holds, every sum $s\equiv7(1-\kappa)\pmod{49}$ is
realizable. There is no condition on the second defects or on
the special cofactor ideal.
\end{theorem}

\begin{proof}
Write $x_j(b)=\gamma_j(\bar c_1(y_bu_b))$ for the generic data.
By (K$_3$) and (K$_4$) of Lemma~\ref{lem:degree}, $x_3,x_4$ are
polynomial functions of degrees at most $3,2$, respectively.
Their top coefficients are $\gamma_{\infty,3}=x_{33}$ and
$\gamma_{\infty,4}=q_2\ne0$, since the special valuation is two.
Summing the generic $\gamma_2$-law in Theorem~\ref{thm:twisted}(a)
and using (K$_2$) and Lemma~\ref{lem:identity} gives
\[
 0=\sum_b\gamma_2(b)=q_2^3+D_0.
\]
The nonzero cubes in $\F_7$ are $\pm1$, proving necessity.

Conversely, choose $q_2$ with $q_2^3=-D_0$ to satisfy the zeroth
moment condition in (K$_2$). By Lemma~\ref{lem:identity}, the first
moment also vanishes if we take
\[
 x_4(b)=q_2b^2,\qquad x_3(b)=x_{33}b^3,\qquad
 x_{33}=-(2q_2)^{-1}D_1.
\]
The remaining moments determine the special coordinates. Set
$\gamma_{\infty,2}=-\sum_b b^2\gamma_2(b)$,
$\gamma_{\infty,3}=x_{33}$ and $\gamma_{\infty,4}=q_2$, and choose
$\gamma_{\infty,1}$ by Theorem~\ref{thm:twisted}(b). To satisfy
(K$_1$), take $x_1(b)=\gamma_{\infty,1}b^5$; the required moments
follow from $\sum_b b^5=0$ and $\sum_b b^6=-1$.
Fullness in Theorem~\ref{thm:twisted}(a) supplies the generic unit
multiples and hence their constant coordinates. The special
constant is free, so take
$\gamma_{\infty,0}=\kappa-\sum_b\gamma_0(b)$ to satisfy the
constant condition. Fullness in part (b) supplies the special unit
multiple. Conditions (K$_5$) and (K$_6$) hold
because $\gamma_5(b)=-1$, $\gamma_{\infty,5}=0$, and all sixth
coordinates vanish. Lemma~\ref{lem:interp} therefore reconstructs
an integer vector for every stated sum, with determinant
$s\,7^9\prod_rN(\mathfrak a_r)$.
\end{proof}

\begin{theorem}[The second exceptional profile]\label{thm:crit11}
Fix the placed cofactor ideals, with generic leading residue $1$.
\begin{enumerate}
\item[(a)] The profile $(1^7,2)$ with coordinate sum of valuation
exactly two is realizable if and only if $D_0=\pm1$.
\item[(b)] The profile $(1^7,3)$ with coordinate-sum class
$\kappa\in\F_7$ is realizable if and only if $D_0=D_1=0$ and
\[
 \begin{cases}
 D_2\ne0, &\text{or}\\
 D_2=0,\ D_3+d_\infty\ne0,\ \kappa\ne1+E_0, &\text{or}\\
 D_2=0,\ D_3+d_\infty=0,\ \kappa=1+E_0.
 \end{cases}
\]
When feasible, every $s\equiv7(1-\kappa)\pmod{49}$ is realizable.
This profile has determinant valuation eleven for $\kappa\ne1$;
the case $\kappa=1$ gives the construction of $7^{12}\Z$ in
Proposition~\ref{prop:witness12}.
\end{enumerate}
\end{theorem}

\begin{proof}
For (a), apply Theorem~\ref{thm:crit10} with $\kappa=1$ and
$s=49u$, $7\nmid u$.

For (b), we first derive the necessary constant condition without
restricting the free coordinates. The special valuation three
forces the quadratic coefficient of $x_4$ to vanish.
Write $x_4(b)=q_0+qb$, and let
$x\ne0$ and $c$ be the cubic and quadratic coefficients of $x_3$.
The two identities in Lemma~\ref{lem:identity} give $D_0=D_1=0$.
Write $D=D_2$ and let $B=\sum_b b x_1(b)$ be the first moment
of $x_1$. The remaining nonconstant conditions give
\[
 \gamma_{\infty,3}=x,\qquad
 r:=\gamma_{\infty,2}=-2xq-D,\qquad
 \gamma_{\infty,1}=-B,\qquad \sum_bx_1(b)=0.
\]
The generic constant law of Theorem~\ref{thm:twisted}(a) sums to
\begin{equation}\label{eq:generic-constant}
 \sum_b\gamma_0(b)
 =-5xq^3+4qB+3q^2D+2cD+2xD_3+E_0.
\end{equation}
Indeed, the contributing sums are
$\sum x_3x_4^3=-xq^3$, $\sum x_1x_4=qB$,
$\sum d_bx_4^2=q^2D$ and $\sum d_bx_3=cD+xD_3$;
$\sum x_4^5=0$. These follow from $D_0=D_1=0$ and the usual
power sums in $\F_7$.
The special law in Theorem~\ref{thm:twisted}(c) gives
\[
 \gamma_{\infty,0}
 =1+2d_\infty x+2r^3x^{-2}+2Brx^{-1}.
\]
Adding the two constant laws and substituting $r=-2xq-D$, we
rewrite the constant condition as
\begin{equation}\label{eq:uniform-eleven}
 \kappa-1-E_0
 =2x(D_3+d_\infty)
  +2D\left(c+\frac{qD}{x}-\frac{D^2}{x^2}-\frac{B}{x}\right).
\end{equation}
Here the only denominator is a power of $x\ne0$. Put
$C=D_3+d_\infty$ and $K=\kappa-1-E_0$, the remaining prescribed
constant. If $D=0$, then the constant condition
\eqref{eq:uniform-eleven} reduces to $K=2xC$, so $K$ and $C$
must vanish together. This gives the necessary alternatives.

For sufficiency, we can set $q=B=0$ and solve for the remaining
coefficients $x,c$. Choose the generic free coordinates as
\[
 x_4(b)=x_1(b)=0,\qquad x_3(b)=xb^3+cb^2,
 \qquad x\ne0.
\]
Then $\gamma_2(b)=d_b$ and
$\gamma_0(b)=1+2d_b(xb^3+cb^2)+e_b$. Set
\[
 \gamma_{\infty,3}=x,\quad \gamma_{\infty,2}=-D,\quad
 \gamma_{\infty,1}=0,\quad
 \gamma_{\infty,0}=1+2d_\infty x-2D^3x^{-2}.
\]
These special coordinates satisfy Theorem~\ref{thm:twisted}(c).
Conditions (K$_1$)--(K$_4$) follow from
$D_0=D_1=0$ and the chosen polynomial degrees; (K$_5$) and
(K$_6$) hold as in the preceding proof. The constant condition is
\begin{equation}\label{eq:explicit-eleven}
 K=2xC+2cD-2D^3x^{-2}.
\end{equation}
If $D\ne0$, then setting $x=1$ determines
$c=D^2+(K-2C)/(2D)$.
If $D=0$ and $C\ne0$, then take $x=K/(2C)$ and $c=0$;
feasibility ensures $x\ne0$. If $D=C=0$, then take $x=1,c=0$,
since $K=0$.
Fullness in Theorem~\ref{thm:twisted} supplies the unit multiples,
and Lemma~\ref{lem:interp} reconstructs a vector for every sum
in the stated class.
\end{proof}


\section{Proof of the classification}\label{sec:closed}

The profile criteria become conditions on rational prime factors
because ideal defects are additive. Since $h_K=1$, every ideal
coprime to $7$ has a generator, and Theorem~\ref{thm:de-props}
makes its defects independent of that generator.
The primes above a rational prime $\ell\ne7$ form one Galois
orbit and have norm $\ell^{f_\ell}$, with $f_\ell=\ord_7(\ell)$.
Conjugation multiplies their defects by $(\tfrac t7)$ and $t^{-1}$.
Thus the unsigned first defect $\delta(\ell)$ and invisibility,
defined in the introduction, are independent of the chosen prime
above $\ell$.

\begin{lemma}[Automatic vanishing, and signs]\label{lem:auto}
Let $\qq\mid\ell$, $\ell\ne7$.
\begin{enumerate}
\item[(a)] If $\ell\equiv3,5,6\pmod7$, then $\ell$ is invisible.
If $\ell\equiv2,4\pmod7$, then $\ee(\qq)=0$.
\item[(b)] If $\dd(\qq)\ne0$, both signs of $\dd(\qq)$ occur
among the primes above $\ell$. If $\ee(\qq)\ne0$, then
$\ell\equiv1\pmod7$, and its conjugates have all six nonzero
second defects.
\end{enumerate}
\end{lemma}

\begin{proof}
The Frobenius $\sigma_\ell$ fixes $\qq$. Hence
\[
 \dd(\qq)=(\tfrac\ell7)\dd(\qq),\qquad
 \ee(\qq)=\ell^{-1}\ee(\qq),
\]
which proves (a) and forces $\ell\equiv1\pmod7$ when the
second defect is nonzero. Conjugation by $\sigma_t$ gives
$(\tfrac t7)\dd(\qq)$ and $t^{-1}\ee(\qq)$. For
$\ell\equiv1$, all six $t$ give distinct primes; for
$\ell\equiv2,4$, the decomposition group is $\{1,2,4\}$ and
the representatives $1,3$ give opposite signs.
\end{proof}

\begin{proof}[Proof of Theorem~\ref{thm:main}]
Theorem~\ref{thm:coprime}, Theorem~\ref{thm:noncoprime} and
Lemma~\ref{lem:pinning} leave only levels ten and eleven outside
$A_7\cup7^{12}\Z$; Proposition~\ref{prop:witness12} below proves
$7^{12}\Z\subseteq S(\GG)$.

\emph{Level ten.} For a value $7^{10}m$, normalization conjugates
the cofactor ideals and gives
\[
 m=\sigma_0\prod_rN(\mathfrak a_r),\qquad \sigma_0=s/7,
 \qquad \sum_b\dd(\mathfrak a_b)=\pm1,
\]
by Lemma~\ref{lem:pinning} and Theorem~\ref{thm:crit10}.
Suppose that $k_\ell$ prime ideals above $\ell$, counted with
multiplicity, occur in the generic cofactor ideals. Their norms
use $k_\ell$ factors $\ell^{f_\ell}$ from $|m|$, and their defects
contribute
$j_\ell\delta(\ell)$, where $|j_\ell|\le k_\ell$ by
Lemma~\ref{lem:auto}. Since $f_\ell k_\ell\le a_\ell$, the
coefficient bounds in (i) follow.

Conversely, take coefficients as in (i), setting $j_\ell=0$
when $\delta(\ell)=0$. Choose $|j_\ell|$ prime ideals above
each remaining $\ell$, allowing repetition, with sign
$\operatorname{sgn}(j_\ell)\delta(\ell)$. Place their product
$\mathfrak a$ on the generic orbit $b=0$, and use trivial ideals
on the other orbits. Then
$D_0=\sum_\ell j_\ell\delta(\ell)=\pm1$ and
$N(\mathfrak a)\mid m$. With
$\sigma_0=m/N(\mathfrak a)$ and $\kappa=1-\sigma_0$,
Theorem~\ref{thm:crit10} realizes $s=7\sigma_0$ and determinant
$7^{10}m$. Unused prime factors go into $\sigma_0$, which
explains why there is no parity restriction on $j_\ell$.

\emph{Level eleven.} Profile $(1^7,2)$ with $v_7(s)=2$ gives
exactly $L_{10}$ by the same argument, using $s/49$ instead of
$s/7$ and Theorem~\ref{thm:crit11}(a). For profile $(1^7,3)$, a
usable visible prime can supply either of the two defects.
Suppose that such a prime $\ell$ divides $m$, choose $\qq\mid\ell$,
and put
\[
 \sigma_0=m/\ell^{f_\ell},\qquad s=7\sigma_0,\qquad
 \kappa=1-\sigma_0\ne1.
\]
If $\dd(\qq)\ne0$, then put $\qq$ on the special orbit and use
trivial ideals elsewhere. Then all $D_k$ and $E_0$ vanish,
whereas $d_\infty\ne0$, so the second branch of
Theorem~\ref{thm:crit11}(b) applies. If $\dd(\qq)=0$ and
$\ee(\qq)\ne0$, then choose a conjugate with second defect
$\kappa-1$ by Lemma~\ref{lem:auto}(b), place it on one generic
orbit, and use trivial ideals elsewhere. The third branch now
applies, since $D_k=d_\infty=0$ and $E_0=\kappa-1$.
Both constructions have determinant
$7\sigma_0\,7^{10}\ell^{f_\ell}=7^{11}m$.

Conversely, every prime ideal in a cofactor has norm dividing
$|m|$, so it lies over a usable rational prime. If all usable
primes are invisible, then additivity makes every defect zero.
Theorem~\ref{thm:crit11}(b) then requires $\kappa=1$, contrary
to valuation eleven; profile $(1^7,2)$ is also impossible
because $D_0=0$. This proves (ii). Condition (i) requires a
usable prime with nonzero first defect, giving
$L_{10}\subseteq L_{11}$. Finally $f_2=3$ and $\delta(2)=3$
(Lemma~\ref{lem:small-prime-types}), so
$8\in L_{11}\setminus L_{10}$ and the inclusion is strict.
All norm factors are positive, and the augmentation can have
either sign. Thus both cofactor sets are symmetric about zero.
\end{proof}

\begin{proposition}[A bounded-count criterion]\label{prop:blocks}
For $m=\pm\prod\ell^{a_\ell}$ coprime to $7$, put
\[
 B_r=\sum_{\delta(\ell)=r}\left\lfloor a_\ell/f_\ell\right\rfloor
 \qquad(r=1,2,3).
\]
Then $m\in L_{10}$ if and only if at least one of
\[
 B_1\ge1,\qquad B_3\ge2,\qquad B_2\ge3,\qquad B_2B_3\ge1
\]
holds.
\end{proposition}
\begin{proof}
For each defect type $r$, the sum of the allowable coefficient
intervals is
$[-B_r,B_r]\cap\Z$. The four alternatives give signed sums
$1$, $3+3=-1$, $2+2+2=-1$, and $3-2=1$ in $\F_7$.
If none holds, then there are no type-$1$ blocks and at most two
type-$2$ blocks or one type-$3$ block, never both types.
None of those signed sums is $\pm1$.
\end{proof}

\begin{proof}[Proof of Corollary~\ref{cor:minima}]
The three smallest prime-ideal norms away from seven are
$8$, $29$, $43$, with unsigned defects $3$, $2$, $1$,
respectively, by Lemma~\ref{lem:small-prime-types}.
Indeed, the first degree-one primes are $29,43$; the least
degree-three norm is $2^3=8$, and the next exceeds $43$;
the least degree-two and degree-six norms are $13^2$ and $3^6$.
Below $43$, one can use at most one norm-$8$ or one norm-$29$
block, and neither satisfies the level-ten criterion. A
norm-$43$ block does, whereas the least visible norm is $8$.
Thus the least positive cofactors in $L_{10}$ and $L_{11}$ are
$43$ and $8$. Since
\[
 43\cdot7^{10}<7^{12}<8\cdot7^{11},
\]
the first is also the least positive seven-divisible value.
The cofactor $1$ has no usable prime, giving the exclusions
$\pm7^{10},\pm7^{11}\notin S(\GG)$.
\end{proof}

\begin{proposition}\label{prop:witness12}
For every $n\in\Z$ there is a vector $v\in\Z^{49}$ with
$s(v)=49n$ and orbit norms $(343,7,7,7,7,7,7,7)$. Consequently,
$7^{12}\Z\subseteq S(\GG)$.
\end{proposition}
\begin{proof}
Take every cofactor ideal to be $R$. All defects vanish, so
Theorem~\ref{thm:crit11}(b), with $\kappa=1$, realizes profile
$(1^7,3)$ with every prescribed sum $s\equiv0\pmod{49}$.
Every cofactor is a global unit and has norm one; hence the
orbit norms are as stated. Taking $s=49n$ gives
\[
 \Theta_\GG(v)=49n\cdot343\cdot7^7=7^{12}n,
\]
including $n=0$.
\end{proof}


\begin{proof}[Proof of Corollary~\ref{cor:sensitivity}]
Fix a positive integer $N$, and put $N_0=N/7^{v_7(N)}$.
By Theorem~\ref{thm:de-props}, choose $z\in R$ coprime to $7$
with $\dd(z)\ne0$. The ideals $7R$ and $N_0R$ are comaximal,
so the Chinese remainder theorem gives $\alpha\in R$ with
\[
 \alpha\equiv z\pmod{7R},\qquad
 \alpha\equiv1\pmod{N_0R}.
\]
Thus $\alpha$ is coprime to $7N$ and $\dd(\alpha)\ne0$.
Additivity of the defects on the prime-ideal factorization of
$(\alpha)$ gives a prime ideal $\qq\mid(\alpha)$ with
$\dd(\qq)\ne0$. The rational prime $\ell$ below $\qq$ therefore
satisfies $\ell\nmid7N$ and $\delta(\ell)\ne0$.

Let $\varphi$ denote Euler's totient function, with $\varphi(1)=1$,
and choose a positive multiple $a$ of $\varphi(N)$ with
$a\ge3f_\ell$. Then $\ell^a\equiv1\pmod N$ and
$\lfloor a/f_\ell\rfloor\ge3$. Each nonzero defect type therefore
satisfies the bounded-count criterion of Proposition~\ref{prop:blocks},
so $\ell^a\in L_{10}\subseteq L_{11}$. Since $1\notin L_{11}$,
the two congruent integers $1$ and $\ell^a$ have different membership
in both cofactor sets. This proves the result for every $N$.
\end{proof}


\section{The scope of the method}\label{sec:larger-primes}

Character factorization and integral reconstruction are available for
$C_p^2$ at every odd prime $p$. The cofactor criteria, however, depend
on two features of the local arithmetic at seven. Both change
already at $p=11$.

First, put $R_p=\Z[\zeta_p]$ and
$A_p=R_p/pR_p=\F_p[\Lambda]/(\Lambda^{p-1})$, with
$\Lambda=\log\zeta_p$. For $z\in A_p^\times$ with constant residue
$a$, write
\[
 \log(z/a)=\sum_{j=1}^{p-2}c_j(z)\Lambda^j.
\]
For a global unit $u$, every conjugate of the algebraic integer
$u/\overline u$ has absolute value one. This quotient is therefore a
root of unity. Since it is congruent to one modulo $1-\zeta_p$, it
is a power of $\zeta_p$. Comparing odd coefficients in its finite
logarithm gives $c_j(u)=0$ for every odd $j$ with $3\le j\le p-2$.
At $p=11$, four independent coordinates
$c_3,c_5,c_7,c_9$ vanish on global units, compared with the two
coordinates $c_3,c_5$ used here. Their simultaneous compatibility
across the character directions requires a new realization argument.

Second, $c_3(\sigma_tz)=t^3c_3(z)$. The cubes in $\F_7^\times$ are
exactly $\{\pm1\}$, which gives both the signed prime-ideal contributions
and the target $\pm1$ in the level-ten criterion. In $\F_{11}^\times$
the cube map is a bijection, since $\gcd(3,10)=1$. The same signed
counting criterion therefore cannot be transferred by replacing
seven with eleven.

These difficulties persist even when every ideal is principal.
For example, $\Q(\zeta_{11})$ has class number one: its Minkowski
bound is less than $59$, and the only prime-ideal norms below the
bound are $11$ and $23$. The corresponding ideals are generated by
$1-\zeta_{11}$ and the conjugates of $1+\zeta_{11}+\zeta_{11}^3$,
whose norms are $11$ and $23$, respectively.
For larger primes, the argument must also determine the full image
of the global units and justify the required principality conditions.
Likewise, the construction of $7^{12}\Z$ uses the valuation-three
ray identity; it does not supply an analogous divisibility threshold
for larger primes.

The shift criterion used for order $25$ remains valid
\cite[Proposition~2.2]{Panraksa2026}, but the required seed values
do not exist at order $49$, since $\pm7^{10}\notin S(C_7^2)$.
A classification for $p\ge11$ therefore needs further arguments.
The obstacles above concern this method; they do not rule out a
classification by other methods.


\section*{Statements and Declarations}
No funding was received for this research. The author declares no
competing interests.

\paragraph{Data availability.}
Exact checks of the displayed prime elements and the finite local
identities are supplied in the ancillary files (Online Resource~1). The classification and
all necessary and sufficient conditions are proved in the article.

\paragraph{Use of generative AI.}
Claude assisted earlier drafting and computational exploration;
OpenAI Codex assisted revision, algebraic calculations, code development,
and verification. The author is responsible for the mathematical
content and accompanying programs.


\providecommand{\bysame}{\leavevmode\hbox to3em{\hrulefill}\thinspace}
\providecommand{\MR}{\relax\ifhmode\unskip\space\fi MR }
\providecommand{\MRhref}[2]{%
  \href{http://www.ams.org/mathscinet-getitem?mr=#1}{#2}
}
\providecommand{\href}[2]{#2}

\end{document}